\documentclass[12pt]{article}
\usepackage{amsmath,amsthm,amsfonts,amssymb,amscd}
\usepackage[latin1]{inputenc}
\usepackage{color}
\usepackage{relsize}

 \newtheorem{theorem}{Theorem}[section]
 \newtheorem{cor}[theorem]{Corollary}
 \newtheorem{lemma}[theorem]{Lemma}
 \newtheorem{prop}[theorem]{Proposition}
 \theoremstyle{definition}
 \newtheorem{definition}[theorem]{Definition}
 
\newtheorem{maintheorem}{Theorem}

 \theoremstyle{remark}
 \newtheorem{rem}[theorem]{Remark}

 \DeclareMathOperator{\slope}{slope}
 \newcommand{\tf}{$\theta$-flat }

 \newcommand{\abs}[1]{\left\vert\,#1\right\vert}

 \newcommand{\norm}[1]{\left\Vert#1\right\Vert}

 \newcommand{\UU}{\mathcal{U}}
 \newcommand{\VV}{\mathcal{V}}
 \newcommand{\CC}{\mathcal{C}}
 \newcommand{\PP}{\mathcal{P}}
 \newcommand{\HS}{\mathcal{H}}

 \newcommand{\ff}{\big(f_{a,\theta}\big)_{a,\theta}}

 \newcommand{\dist}{\operatorname{dist}}
 \newcommand{\ang}{\operatorname{angle}}
 \newcommand{\ve}{\textbf{e}}
 
 \newcommand{\vt}{\textbf{t}}
 
 \newcommand{\Leb}{\operatorname{Leb}}

 \newcommand{\wzero}{\text{\bf w}_0}
 \newcommand{\wn}{\text{\bf w}_n}
 \newcommand{\wj}{\text{\bf w}_j}
 \newcommand{\wjp}{\text{\bf w}_{j-1}}
 \newcommand{\wnp}{\text{\bf w}_{n-1}}
 \newcommand{\wpp}{\text{\bf w}_{p}}
 \newcommand{\bfw}{\text{\bf w}}

\usepackage{psfrag}
\usepackage[dvips]{graphicx}
\DeclareGraphicsExtensions{.eps}

\title{Non-periodic bifurcations for surface diffeomorphisms}
\author{Vanderlei Horita, Nivaldo Muniz and Paulo Rog\'{e}rio Sabini}
\date{\today}

\begin{document}

\maketitle

\begin{abstract}
We prove that a ``positive probability" subset of the boundary of the
set of hyperbolic (Axiom A) surface
diffeomorphisms with no cycles $\HS$ is constituted
by Kupka-Smale diffeomorphisms: all periodic points are hyperbolic and
their invariant manifolds intersect transversally.
Lack of hyperbolicity arises from the presence of a tangency between a
stable manifold and an unstable manifold, one of which is {\em not}
associated to a periodic point.
All these diffeomorphisms that we construct lie on the boundary of the
same connected component of $\HS$.
\end{abstract}

\let\thefootnote\relax
\let\thefootnote\relax
\let\thefootnote\relax\footnote{Work partially supported by CAPES, CNPq, FAPESP, INCTMat and
PRONEX.}

\section{Introduction}

One of the most challenging problems in Dynamical Systems theory is
to understand how the stability breaks down under small changes of
the evolution law.
We say that a system is {\em stable} if there exists a neighborhood
where all systems are topologically conjugated to it.
It is a well known fact that hyperbolic (Axiom A) systems with no
cycles are stable.
A very successful method to study the breakdown of hyperbolicity is
by considering parametrized families of systems starting inside the
hyperbolic domain and describing the ways how hyperbolicity is
destroyed when the parameter varies.
Many authors have studied this problem following this approach, see
\cite{BDV98,DHRV09,En98,HMS07,MY1,NPT76,PT85,PT87,PV94,PY94,So73a},
just to mention a few references.
In all these works the mechanism responsible for the break down of
stability involves periodic orbits; indeed, it falls in one of the
following two types\,:
\begin{description}
\item{(NH)} {\em there exists an unique periodic orbit that is
non-hyperbolic,} and it is either a saddle-node (one eigenvalue equal
to $1$), a period-doubling (one eigenvalue equal to $-1$), or a Hopf
orbit (two complex conjugate eigenvalues with norm $1$);

\item{(NT)} all the periodic orbits are hyperbolic, but {\em there
exists an unique non-transverse intersection between some stable and
some unstable manifold of periodic orbits}; this intersection is
quasi-transverse (codimension $1$).
\end{description}

Newhouse and Palis in \cite{NP76} conjectured that (NH) and (NT)
are the generic mechanism for the collapse of hyperbolicity along
families of diffeomorphisms starting from a Morse-Smale
diffeomorphism.
That is, generically, Morse-Smale diffeomorphisms should remain
hyperbolic for as long as they remain Kupka-Smale.
Newhouse, Palis, and Takens in \cite{NPT83} show that this conjecture
is true when the limit set (the set of forward and backward limit points)
is still finite at bifurcation parameter\,: generically the bifurcating
diffeomorphisms is of type (NH) or (NT).
To the best of our knowledge there has been essentially no other
significant progress in the direction of this conjecture.

Motivated by the advances in the theory of Hénon-like dynamics, Bonatti,
Viana suggested that the problem should be approached from a
probabilistic point of view, and, in this context, the conclusion should
be opposite.
They conjectured that there exists a subset of the boundary of the set
of hyperbolic systems which is formed by Kupka-Smale diffeomorphisms
and has positive measure, in some natural sense, so that (NH) and (NT)
should not account for {\em almost all} transitions to non-hyperbolicity.
The conjecture was proved by the present authors \cite{HMS07} in the
setting of non-invertible circle maps.
In the present work we prove that the Bonatti-Viana conjecture is also
true for surface diffeomorphisms.

In all what follows $M$ will denote a surface
and $\HS$ the set of hyperbolic (Axiom A) systems with the no-cycle condition
defined on $M$.

\begin{maintheorem}
\label{teo.main}
There is an open set $\mathcal{U}$ of $2$-parameter families $\ff$
of surface diffeomorphisms such that for a positive set
$\mathcal{A}$ of parameters $a$
\begin{itemize}
\item[(a)] for some $\theta^* = \theta^*(a)$, the map $f_{a,\theta^*}$
has a heteroclinic (cubic) tangency between invariant manifolds, one of
which is not associated to a periodic point;
\item[(b)] all $f_{a,\theta^*}$ belong to the boundary of $\HS$.
\item[(c)] all periodic points of $f_{a,\theta^*}$ are hyperbolic;
\item[(d)] all intersections between stable and unstable manifold of
periodic points of $f_{a,\theta^*}$ are transverse;
\end{itemize}
\end{maintheorem}

In fact, the diffeomorphisms $f_{a,\theta^*}$ that we construct are
in the boundary of the same connected component of $\HS$
(their lack of hyperbolicity
follows from item (a)).
The corresponding property for circle maps could not be proved in
\cite{HMS07}.
We present here a proof of that fact,
thus strengthening Theorem A in
that paper.

\begin{maintheorem}[Theorem A of \cite{HMS07} revisited]
\label{t.teo_a}
There exists an open set $\mathcal{U}$ of $2$-parameters families
$(f_{a,\theta})_{a,\theta}$ of maps of the circle such that for some
$\theta_* = \theta_*(a)$, the map $f_{a,\theta_*}$ has a (cubic) critical
point.
Moreover, for a positive Lebesgue measure set $\mathcal{A}$ of
parameters $a$:
\begin{itemize}
\item[(1)] there exists a continuous curve $a(\theta)$ in the
parameter space $(a,\theta)$, with $a(\theta_*) \in \mathcal{A}$,
such that $f_{a(\theta),\theta}$ belongs to the interior of the
uniformly hyperbolic (expanding) domain for every $\theta <
\theta_*$;
\item[(2)] all periodic points of $f_{a,\theta_*}$ are hyperbolic
(expanding), and no critical point is pre-periodic.
\end{itemize}
\end{maintheorem}

The diffeomorphisms that we consider in Theorem~\ref{teo.main}
are (strongly) dissipative, indeed we view then as a kind of singular
perturbation of the cubic maps \cite{HMS07} in much the same way as
Hénon-like diffeomorphisms are treated as perturbations of quadratic
maps of the interval 
 in works as \cite{BC91}, \cite{MV93},
see Remark~\ref{r.one2two}.

We adapt techniques developed by Benedicks, Carleson \cite{BC91} and
Mora, Viana \cite{MV93} in those papers to obtain exponential growth
of the derivative in some direction.
An important new difficulty arises from the fact that the tangency
(criticality) for the ``unperturbed" circle maps is degenerate (cubic).
This makes it specially tricky to detect and control tangencies for the
kind of singular perturbation that we deal with, all the more so because
we also have to deal with stable manifolds not associated to any
periodic point.
A key ingredient to circumvent this difficulty is to establish a good
notion of critical points and appropriated control of degeneration of
angles between tangent direction of invariant manifolds.

The core of this paper comes from Paulo Sabini's doctoral thesis. We
considered it worthwhile to work on that initial text, sharpening the
results and improving the arguments, and that led to the present joint
paper. Sadly, during that work, Paulo passed away, at the age of 33.
We miss him deeply.

\medskip

{\it Acknowledgements.} V. Horita and N. Muniz are grateful to M. Viana
by bring the authors together in this subject and by discussion of mathematical
ideas in this paper, especially after P. Sabini passed away.
It was crucial the encouraging of Marcelo for the completion of this
article.

\section{Proof of Theorem~\ref{t.teo_a}}

\begin{proof}

The claim is that the families present in Theorem A of \cite{HMS07}
satisfy both items. The second one has already been proved so we just need  to prove
item (1).

Recall the definition of the families in \cite{HMS07}.
We can suppose, by reparametrizing those families, that $\theta_*(a) =
0$ for all $a$. The construction of the set $\mathcal{A}$ is based
on the method present in \cite{BC85} for quadratic maps. Roughly
speaking, we exclude parameters $a$ that infringe a given restriction
of the loss of expansion due to the proximity of a point that we call
the {\em critical point}.
Although in our context we do not have a critical point for negative
$\theta$, we introduce a notion that has the same role.
In \cite{HMS07}, for each fixed parameter $\theta \le 0$ we perform
an exclusion of parameters $a$
in order to obtain a positive measure subset $\Omega_\theta$ of
parameters $a$ such that the orbit of the critical point presents
expansion of the derivative. In fact, we do not take into account the
fact that our \emph{critical point} is not  a true criticality
(zero derivative). At the end of the construction we obtain a
set $\Omega_\theta$ with empty interior, but positive measure.
Here, we intend to revisit those techniques in a more accurate way for showing
that it is possible to get $\Omega_\theta$ as an union
of a finite number of intervals of parameters. The reason why this is
expected is because after a certain step $n = n(\theta)$ of the
construction we do not have to exclude parameters anymore
since the derivative is uniformly bounded
away from zero.

The exclusions of $a$-parameters for a fixed $\theta < 0$
take place for two
reasons: the first one is to control the recurrence of returns of
the critical point, and the second one is to avoid too frequent
recurrence to the critical region $(a-\delta,a+\delta)$. The
heuristic to circumvent these two mechanisms of exclusion of
parameters follows.
Let $\varepsilon = f_{a,\theta}'(a) > 0$ be the derivative of
$f_{a,\theta}$ at the critical point $a$.
By construction, $a$ is taken close to the distinct pre-image
$\bar{a}$ of the fixed point $p$ of the initial map $f$, see
\cite[Section~1.4]{HMS07}. We address the reader to
follow this discussion in parallel to Section 2 of \cite{HMS07}.
In particular in the sequel $c$ stands for the expansion
of derivatives stated therein.
Let $J$ be a small neighborhood of $p$ such that
$\sup\{f_{a,\theta}'(x) \colon x \in J \} \ge \sigma > 1$.
Fix an integer $m_0>0$ satisfying $\varepsilon \cdot
\sigma^{m_0} > e^{\tilde{c}\,m_0}$, for some constant $\tilde{c}>0$.
We can suppose (by reducing $\eta$ in Proposition 2.3 of
\cite{HMS07}, if necessary), that $f_{a,\theta}^j(a)
\in J$ for every $1\le j \le m_0$.
We can even suppose that $m_0$ is large enough implying
$\tilde{c}>c$, since $\sigma > e^c$.
By continuity, there exists $\zeta >0$ such that for every
$1\le j \le m_0$, we have
$$
f_{a,\theta}^j(x) \in J \qquad \text{for every } x \in (a-\zeta,
a+\zeta).
$$
Let us fix $\hat{c} < c$ with $\abs{\hat{c}-c} \ll 1$ and
let $\tilde{n} \geq 1$ be such that
$e^{-\tilde{n} \alpha} < \zeta$ and $\varepsilon e^{c\tilde{n}} \geq e^{\hat{c}(\tilde{n}+1})$.
Let $\omega \in \Omega_{\theta,\tilde{n}}$ and
Let $n_k$ be the first return situation for $\omega$ after $\tilde{n}$.
\begin{align}
&|(f_{a,\theta}^j)'(f_{a,\theta}(a))| \ge e^{c\,j}, \qquad \text{for } 1\le j \le n_k-1.\\
\label{sgda}
&|(f_{a,\theta}^j)'(f_{a,\theta}(a))| \ge \varepsilon e^{c\,(j-1)}e^{\tilde{c}(j-n_k)} \ge e^{\hat{c}\,j}, \qquad \text{for } n_k \le j \le n_k+m_0-1.\\
&|(f_{a,\theta}^j)'(f_{a,\theta}(a))| \ge e^{c\,n_k}e^{\tilde{c}\,m_0} \ge e^{c\,(n_k+m_0)} \qquad \text{for } j = n_k+m_0.\\
&|(f_{a,\theta}^j)'(f_{a,\theta}(a))| \ge e^{c\,j} \qquad \text{for } n_k+m_0+1 \le j < n_{k+1}.
\end{align}

Hence every return situation
 after $\tilde{n}$
can be dealt with without exclusions for $(BA)_n$ violations and every parameter $a$
not excluded up to that time will satisfy:
\[
   |(f_{a,\theta}^j)'(f_{a,\theta}(a))| \ge e^{\hat{c}\,j} \qquad \text{for }j\geq 1
\]

%
%

Thus,
$$
\Omega_\theta = \lim_{n\to\infty}\Omega_{\theta,n} = \Omega_{\theta,\tilde{n}}.
$$

Recall that for each $n$ the set $\Omega_{\theta,n}$ is the union of a
finite number of intervals of parameters $a$, see
\cite[Section~4]{HMS07}. So, in each vertical line over $\theta$,
$-1 \le \theta \le 0$, we have a finite union of intervals
(see Figure~\ref{f.blabla}).
Moreover, for $\theta$ close to $-1$, we do not need to exclude
any parameter, because the initial map is uniformly expanding.
When $\theta$ increases, the derivative decreases in the perturbation
region and some exclusion of parameters is necessary due the loss
of expansion. Increasing $\theta$ we have to exclude more
and more intervals of parameters $a$. Since the whole structure varies
continuously with $\theta$, we have \emph{legs} in the
$(a,\theta)$-plane such that all maps outside them are uniformly expanding (see
Figure~\ref{f.blabla}).

\begin{figure}[phtb]
\centering
\psfrag{tn}{$\theta_0$}
\psfrag{l}{$legs$}
\psfrag{a}{$a$}
\psfrag{t}{$\theta$}
\includegraphics[height=5cm]{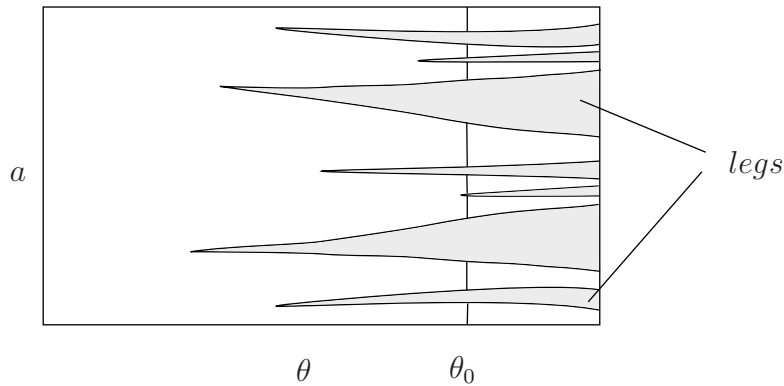}
\caption{Parameter space for a fixed $n$} \label{f.blabla}
\end{figure}

As we see above, fixed a negative $\theta_0$ close to $0$ there is
$\tilde{n}(\theta)$ (which can be supposed increasing with $\theta$) such
that for all $n \ge \tilde{n}$ we do not need to exclude more parameters of
$\Omega_{\theta,n}$, for every $\theta < \theta_0$.
Thus, the parameters space do not change in the rectangle
$[-1,\theta_0] \times [a-\eta,a+\eta]$.

Finally, when $\theta_0$ goes to $0$ and $n$ tends to infinity, for each
$a \in \mathcal{A}$ we have a continuous path
$$
a(\cdot):[-1,0] \to [-1,0) \times [a-\eta,a+\eta]
$$
inside the parameter space such that each corresponding map is
uniformly expanding. This completes the proof of Theorem~\ref{t.teo_a}.
\end{proof}

\section{The families}\label{s:fam}

Here we describe the families which are object of
Theorem~\ref{teo.main}. In the following sections we prove
that they satisfy the claims.

Let $f$ be a $C^r$-diffeomorphism, $r\ge 3$ defined on a compact
boundaryless surface $M$ having a {\it strongly dissipative non
trivial minimal attractor} $\Lambda$, that is, $\Lambda$ is a
minimal, (uniformly) hyperbolic, transitive (it has a dense orbit),
and attracting set ($\Lambda = \cap_{n \in \mathbb{N}}
f^n(\mathcal{U})$ for some neighborhood $\mathcal{U}$ of $\Lambda$)
where the contraction of the stable bundle is much stronger than the
expansion of the unstable bundle in a sense to be make precise in a
little while.
So, $\Lambda = \overline{W^u(P)}$ for every periodic point $P$ of $f$.

Let $P$ and $Q$ be periodic points in $\Lambda$ with distinct
orbits, and let $q \in \Lambda$ be a heteroclinic point such that
$q \in W^u(P) \cap W^s(Q)$.
For simplicity we suppose that both $P$ and $Q$ are fixed
points of $f$.

Assuming a finite number of nonresonance properties on the
eigenvalues of $P$ and $Q$ we can assume that $f$ is linearizable at
$P$ and $Q$ and that the linearizing coordinates vary continuously
on a small $C^3$-neighborhood $\VV$ of $f$, see \cite{St58},
\cite{BDV98}, and \cite{En98} for details.
For every $g \in \VV$, we denote $(x_i,y_i)$, $i=1,2$ such coordinates
on neighborhoods $\UU_i$ of $P$ and $Q$, respectively (we omit the
dependence of the coordinates on the diffeomorphism, for simplicity
of notation), see Figure~\ref{f.coordenadas}.

\begin{figure}[phtb]
\centering
\psfrag{P}{$P$}
\psfrag{Q}{$Q$}
\psfrag{q}{$q$}
\psfrag{r}{$\hat{q}$}
\includegraphics[height=6.5cm]{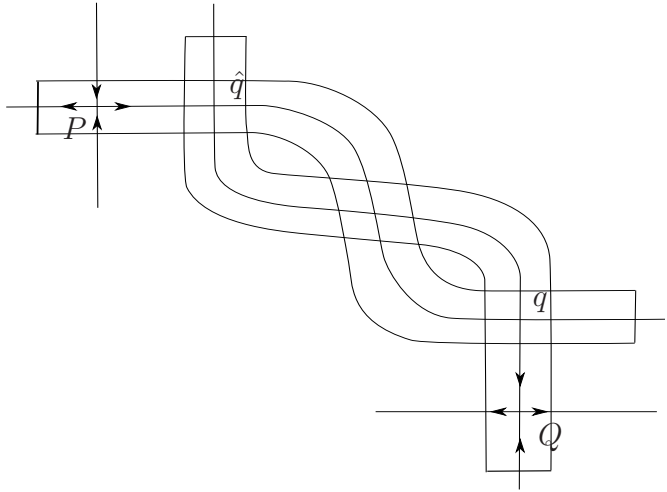}
\caption{Linearizing coordinates} \label{f.coordenadas}
\end{figure}

In this system of coordinates, points $z = (x_1(z),y_1(z))$ in
$W^s_{loc}(P)$ (resp. in $W^u_{loc}(P)$) are such that $x_1=0$
(resp. $y_1=0$), and similarly points $z=(x_2(z),y_2(z))$ in
$W^s_{loc}(Q)$ (resp. in $W^u_{loc}(Q)$) are such that $x_2=0$
(resp. $y_2=0$). By extending the neighborhoods $\UU_i$, if
necessary, we can suppose that $\UU_1 \cap \UU_2$ contains $q$ and
$\hat{q} = f^{-1}(q)$. The previous systems of coordinates
$(x_i,y_i)$ define in a neighborhood $V$ of $q$ and $\hat{V}$ of
$\hat{q}$, respectively, coordinates $(x,y)$ and
$(\hat{x},\hat{y})$, where $x$ and $\hat{x}$ are given by $x_2$, and
$y$ and $\hat{y}$ are given by $y_1$. In $V$, the expression of any
diffeomorphism $g \in \VV$ in these coordinates is linear:
$$
x(g(z)) = \sigma_1 \cdot \hat{x}(z) \quad \text{ and } \quad
y(g(z)) = \lambda_2 \cdot \hat{y}(z),
$$
where $\lambda_1$ and $\sigma_1$ (resp. $\lambda_2$ and $\sigma_2$)
are eigenvalues of $P$, (resp. $Q$).
It follows from the assumption of strong dissipativeness that
$|\lambda_i| \ll 1 < |\sigma_i|$.
Let us write $\lambda_2 = b$ in analogy to parameters of Hénon-like
families ($b \ll 1$).
For simplicity of notation, let us denote both coordinates on
$\hat{V}$ and $V$ as $(x,y)$.

We consider a $2$-parameter $C^r$-family $\tilde{f}_{a,\theta} : M
\to M$, $a \in [-\eta,\eta]$ and $\theta \in [-\eta,\eta]$ as follows.

\begin{itemize}
\item[$(H_1)$] For every $a$, the map $\tilde{f}_{a,\theta_0}$
is $C^r$-close to $f$ (and so it is uniformly hyperbolic).

\item[$(H_2)$]
There is an open rectangle $R_0 \subset \hat{V}$ containing $\hat{q}$ such
that, for each $a$ and $\theta$, the maps
$\tilde{f}_{a,\theta}$ and $f$ are $C^r$-close outside $R_0$.
\end{itemize}

We choose $\delta > 0$ sufficiently small such that the square
$R_\delta = [-\delta, \delta] \times [-\delta, \delta]$ centered in
$\hat{q}$ is contained in $R_0$, see also Section~\ref{ss:mane}.
We write $S = \tilde{f}_{a,\theta_0} (R_\delta)$.
Let us assume that $\eta>0$ small enough in order to have $q+(a,0) \in S$,
for every $a\in [-\eta,\eta]$.
\begin{itemize}
\item[$(H_3)$]
We deform $\tilde{f}_{a,\theta_0}$ inside $R$ in such a way that
$\tilde{f}_{a,\theta}$ has local form in $R_\delta$ given by
$$
\Phi_{a,\theta}(x,y) = (a +  \sigma_1 y +
x(- A_1 \theta + B_1 x^2 + C_1 y^2),
- b x + y(- A_2 \theta + B_2 x^2 + C_2 y^2)),
$$
where
$A_i,B_i,C_i$ positive constants with $A_2,B_2,C_2 \le K b$, for some constant $K>0$.
See Figure~\ref{f.perturbation}.

\item[$(H_4)$] We choose $B_1 \ge C_1 \ge 4\sigma_1 \delta^{-2} \ge 4A_1 > 0$ and
$C_1 \le 5\sigma_1 \delta^{-2}$ (see Remark~\ref{r.h4} and Section~\ref{ss:mane}),
such that for every $z\in R_0 \setminus R_\delta$ and all norm-1 vector $v=(v_1,v_2)$
with $\slope(v)=\abs{v_2}/\abs{v_1} < (1/10)$ we have
\begin{itemize}
\item[$\bullet$] $ \slope (D\tilde{f}_{a,\theta}(z)\cdot v) \le 1/10$; and \\
\item[$\bullet$] $\| D\tilde{f}_{a,\theta}(z)\cdot v \| > \sigma_0 > 1$,
\end{itemize}
for all $\abs{a} < \eta$ and $\abs{\theta} < \delta^2$. (See Remark~\ref{r.h4}).
\end{itemize}
\begin{figure}[phtb]
\centering
\psfrag{P}{$P$}
\psfrag{Q}{$Q$}
\psfrag{R}{$R$}
\psfrag{S}{$S$}
\psfrag{r}{$\hat{q}$}
\psfrag{q}{$q$}
\psfrag{f}{$\tilde{f}_{a,\theta}(\hat{q})$}
\includegraphics[height=4.5cm]{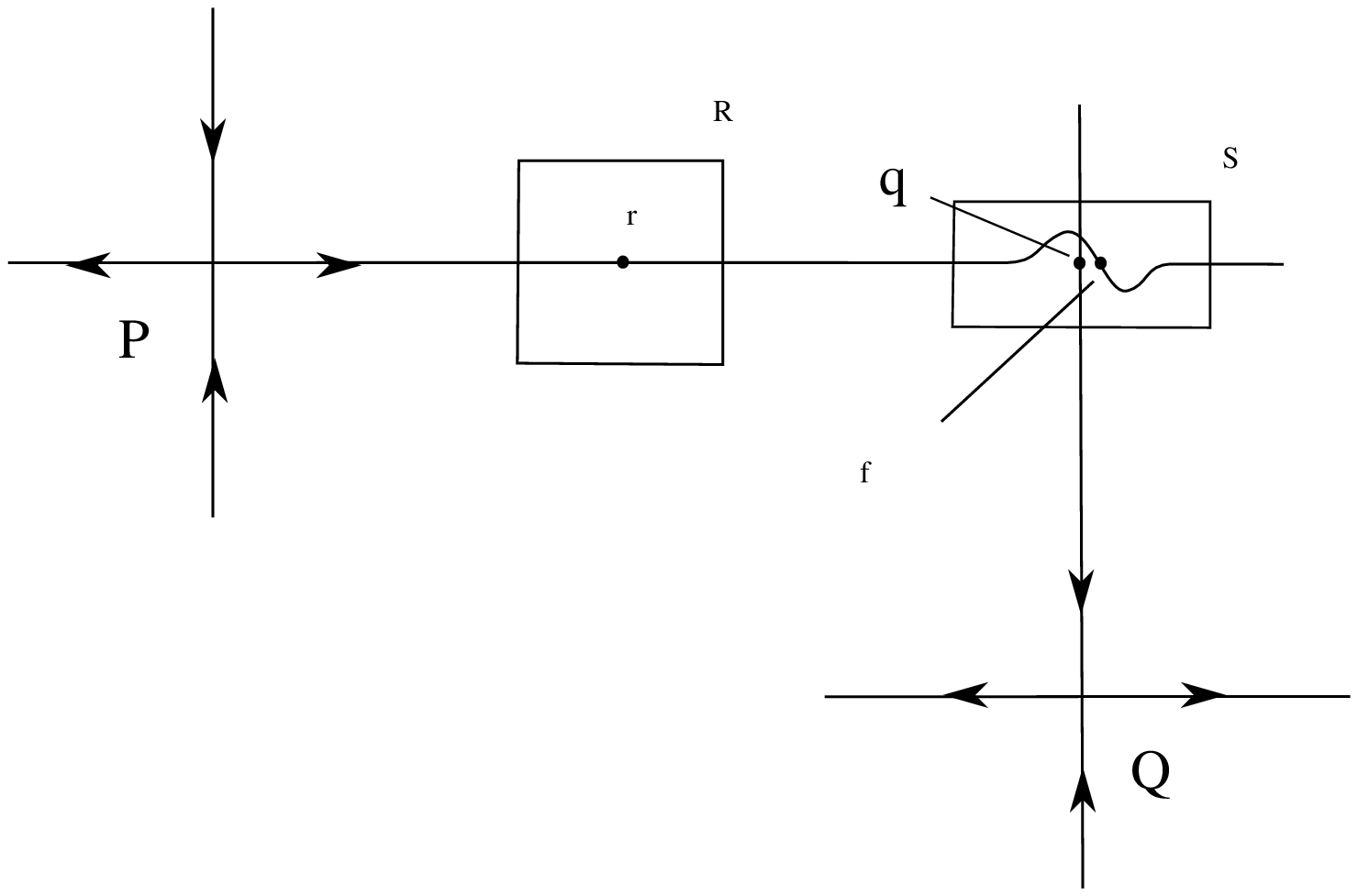}
\caption{Local deformation} \label{f.perturbation}
\end{figure}

Notice that the perturbation is made  in such a way that for every
$\theta < 0$ the angle between the image of horizontal directions by the
derivative of the local form of $\Phi_{a,\theta}$ and the vertical
direction is (uniformly) bounded away from zero.
Furthermore, for $\theta = 0$, the image of the horizontal direction by
$\Phi_{a,\theta}$ and the vertical direction has an unique point of (cubic)
tangency at $q + (a,0)$, see Figure~\ref{f.cubic-local}.
So, for each $a$ there exists $\theta_* = \theta_*(a)$ very close to $0$
such that the slope of the image of a horizontal directions by the
derivative of $\tilde{f}_{a,\theta}$ is (uniformly) bounded away from zero,
for every $\theta < \theta_*$; and, for $\theta = \theta_*$ the image of
the horizontal direction and the vertical direction has an unique point
of (cubic) tangency.
So, we can assume that $[\theta_0,\theta_1] = [-1,\varepsilon]$,
$\varepsilon > 0$.
We remark that, in principle, the tangency has no dynamical meaning,
since the vertical direction may not be part of a contractive bundle.
\begin{figure}[phtb]
\centering
\psfrag{R}{$R$}
\psfrag{r}{$\hat{q}$}
\psfrag{l}{$\Phi_{a,\theta}$}
\psfrag{q}{$q+(a,0)$}
\includegraphics[height=4cm]{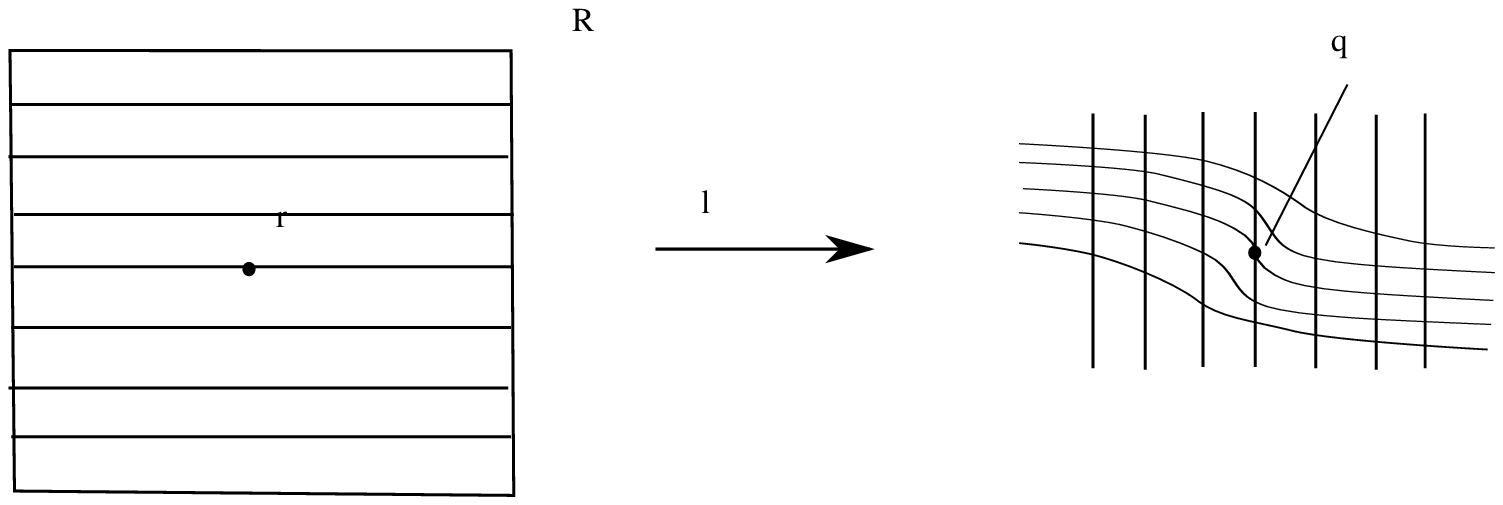}
\caption{Cubic tangency} \label{f.cubic-local}
\end{figure}

According $(H_1)$ all families $(\tilde{f}_{a,\theta})_\theta$
start inside the set of uniformly hyperbolic diffeomorphisms. The
horizontal segment through $\hat{q}$, is mapped by
$\tilde{f}_{a,\theta_0}$ in a curve that makes an angle with the
vertical direction that decreases when $\theta$ increases.
When $\theta = \theta^*(a)$, this curve tangencies the vertical
direction at $q+(a,0)$.
Although the vertical direction has no dynamical meaning, if there are
directions of a contractive bundle that are almost vertical defined
during all process then the 1-parameter family unfolds a cubic tangency
between the unstable manifold of $P$ and a stable manifold of some
point.
Nevertheless, there are no reasons for the existence of those
contractive direction during all the process (and, in general, it seems
that they do not exist). An additional problem is to determine if those
direction, when exist, are not associated to a periodic point. After
overwhelm these two steps, it remains to show that this is the first
bifurcation of the family.

\smallskip

\begin{rem}
\label{r.h4}
Let us comment on the compatibility of assumption $(H_4)$ with
our construction.
Given a norm-1 vector $v = (v_1,v_2)$ with $\slope(v) \leq 1/10$
and $z \in \partial R$, we have
$$
\slope D\Phi_{a,\theta}(z)\cdot v < 1/15
$$
and
$$
\| D\Phi_{a,\theta}(z)\cdot v \| > 2.
$$
\end{rem}
\begin{proof}
Given $z = (x,y)$, from $(H_3)$, we have
\begin{align}
\label{e.v0}
D\Phi_{a,\theta}(z) \cdot v = \big((-A_1\theta + 3B_1 & x^2 + C_1 y^2)v_1 +
(\sigma_1 + 2 C_1 xy)v_2, \\
& (-b + 2B_2xy)v_1 + (-A_2\theta + B_2 x^2 + 3C_2 y^2)v_2\big). \nonumber
\end{align}
For some positive constant $\hat{K}$,
\begin{equation}
\label{e.v1}
\abs{(-b + 2B_2xy)v_1 + (-A_2\theta + B_2 x^2 + 3C_2 y^2)v_2} \le \hat{K}b.
\end{equation}
Since $\norm{v}=1$ and $\abs{\slope(v)} \leq 1/10$ we get $\abs{v_1} > 9/10$.
Moreover, from $\abs{\theta} < \delta^2$ we obtain
\begin{align*}
|\, (-A_1\theta + 3B_1 x^2 & + C_1 y^2)v_1 +
(\sigma_1 + 2 C_1 xy)v_2 \,| \geq  \\
& \ge \abs{v_1}[-A_1\abs{\theta} + 3B_1 x^2 + C_1 y^2 -
(\sigma_1 + 2 C_1 \abs{xy})\abs{v_2/v_1}] \\
& \ge (9/10)(-A_1 \delta^2 + 3B_1 x^2 + C_1 y^2 - (\sigma_1 + 2 C_1 \abs{xy})(1/10) ).
\end{align*}
Furthermore points $(x,y) \in \partial R$ have either $\abs{x} = \delta$ or
$\abs{y} = \delta$, then as $B_1 \ge C_1$,
\begin{align*}
\label{e.h4}
|\, (-A_1\theta + 3B_1 x^2 & + C_1 y^2)v_1 +
(\sigma_1 + 2 C_1 xy)v_2 \,| \geq  \\
& \ge (9/10)[-A_1 \delta^2 + C_1 \delta^2 - (\sigma_1 + 2C_1 \delta^2)/10] \\
& \ge (9/10)[-\sigma_1 + 4\sigma_1 - (11/10)\sigma_1] > 2 \sigma_1.
\end{align*}
Then, from \eqref{e.v0}, \eqref{e.v1}, and last estimate, we obtain
$$
\abs{\slope (D\Phi_{a,\theta}(z)\cdot v)} \le \frac{\hat{K}b}{2\sigma_1} < 1/15.
$$
and
$$
\| D\Phi_{a,\theta}(z)\cdot v \| \ge 2\sigma_1 > 2.
$$

We conclude the proof of the remark.
\end{proof}

The dynamics outside $R_0$ is unchanged (so hyperbolic) for every
$\tilde{f}_{a,\theta}$, by $(H_2)$.Thus, it is possible to define $\tilde{f}_{a,\theta}$
in $R_0 \setminus R$ as in condition $(H_4)$.

Our results apply to every family $(f_{a,\theta})_{a,\theta}$ in
a small $C^r$-neighborhood $\mathcal{U}$ of
$(\tilde{f}_{a,\theta})_{a,\theta}$.

%
%

\medskip

Related to this setting of cubic tangencies, in \cite{BDV98,En98},
the authors have constructed codimension-3 submanifolds (respectively,
codimension-2 submanifolds) of the border of the set of Anosov
$C^r$-diffeomorphisms of the torus $\mathbb{T}^2$ (respectively, surface
diffeomophisms with basic set different of all surface) corresponding to
existence of a cubic tangency between the stable and unstable manifolds
of a pair of periodic points.
Roughly speaking, they deform a family of diffeomorphisms in a neighborhood
of a heteroclinic point in order to create a tangency.
During all the process the family remains Anosov.
Those families can be adapted to our context by considering a parameter $a$
defining where the tangency with the vertical direction is being created (in those
context stable directions of periodic saddles correspond to the vertical direction).
See Figure~\ref{f.perturbation}.

Let us point out some distinguishing characteristics of the present setting.
The bifurcation in \cite{BDV98,En98} are of (NT) type, i.e. the lack of hyperbolicity
is due to the presence of a non-transverse intersection of invariant manifolds
of periodic points.
Since the region of perturbation does not contain none of the periodic points $P$ and $Q$,
the arc of $W^u(P)$ (respectively $W^s(Q)$), from $P$ through a neighborhood of $q$
(respectively from $Q$ through a neighborhood of $q$) is always defined and controlled.
When a tangency is created between invariant manifolds of periodic points there is no
recurrence of the tangency to the region of perturbation, opposite to the present
setting\,: we are going to control the creation of a tangency which is recurrent.
Another feature of our setting is that we have to ensure that there exist sufficiently
many branch of stable manifolds close to $q$.
After all, we obtain a tangency between $W^u(P)$ and a stable manifold not associated to a
periodic point.

\begin{rem}
  \label{r.one2two}
The circle maps in \cite{HMS07}, after a convenient reparametrization, have a
local form (see \cite[Remark 2.2]{HMS07}) given by
$$
x \mapsto a + A\theta x + B x^3 + h.o.t.
$$
The reader can realize that the analogy between this formula and that of $(H_3)$
resembles the respective analogy between the quadratic family $x \mapsto 1-ax^2$ and the
H\'enon family $(x,y) \mapsto (1-ax^2 + y, bx)$.
%
\end{rem}



\section{Overview of quadratic and Hénon systems}

As mentioned before our methodology
is based on techniques grounded on works of
Benedicks, Carleson, Mora, and Viana \cite{BC85,BC91,MV93,Vi93}.
In fact we adapt and extend results present in \cite{HMS07}
recreating at some extent the parallels
between the one-dimensional quadratic family \cite{BC85}
and the two-dimensional Hénon \cite{BC91} and Hénon-like \cite{MV93} families.

We refer the reader to the excellent survey of all these classical
arguments present in \cite{LV03} which encompass a
study guide to the original papers.

Although we need to focus on the distinctive features of our setting
contrasting with Hénon-like families it is worthwhile
recall some of the key aspects of the original arguments. We follow
closely \cite{LV03}.

\paragraph{Uniform expansion outside a critical region.}

A basic result which bounces to Mañé's works gives uniform hyperbolicity for (pieces of) orbits
avoiding neighborhoods of critical points and periodic attractors.
First of all we need to define
an appropriate {\em critical region}. These notions are materialized
in Section~\ref{ss:mane}.

\paragraph{Bounded recurrence and non-uniform expansivity.}
An initial set of parameters is fixed and a concept of
``good'' parameters is built upon a careful designed
control of recurrence of some orbits to the critical region.
More precisely, orbits of points in a special set $\CC$
are studied aiming two assumptions which must be satisfied for all $n \geq 1$:

(BA) There is an exponentially decreasing (with $n$) lower bound  for  the
recurrence depth of the $n$-iterate of such points; and

(FA) There is an exponentially decreasing (with $n$) frequency  of recurrence.

Formalizations of these notions in the cubic setting are presented in
Section~(\ref{ss:ind}).

\paragraph{Dynamically defined critical points.}

In the course of extending the one-dimensional arguments,
a highly relevant conceptual point is the
non existence of actual critical points for the families (of diffeomorphisms)
considered.
An original contribution of \cite{BC91} at this point
is to identify tangencies between stable and unstable leaves
as natural substitutes for the notion of critical points.
At the same time it becomes necessary overwhelm several and new complications.
Since the precise formulation
of the critical points needs knowledge a priori of the positions of stable leaves,
assumption clearly not realistic,
the argument has to be settled
introducing the auxiliary notion of {\em finite time critical approximations}.
Furthermore, the recovery
argument after revisits of some neighborhood of the criticalities
must be encompassed this time taking into account
geometric complications introduced by the
two-dimensional scenario and also
the need to satisfy two apparently contradictory forces.
On the one hand, we need to
have a sufficiently rich set of critical approximations from which one can gets
inductive information about growth
of derivatives.
On the other hand, each critical approximation imposes the same parameter
exclusion rules related to the depth and frequency of recurrence as present
in the one-dimensional case. Section (\ref{ss:stabf})
present the related notion of fields of contractive directions
and Section (\ref{ss:cp}) is devoted to discuss the issues
associated to the notion of critical points for families in Section~\ref{s:fam}.

\paragraph{The induction.}

We start with an initial interval  $\Omega_0$ of parameters and
wish to build a positive Lebesgue measure $\Omega = \Omega_\infty$
subset of $\Omega_0$ containing parameters for which

(EG) $\qquad \abs{Df^n(f(z)).(1,0)} \geq e^{cn},\text{ for all } n
\geq 1 \text{ and all } z \in \CC$.

In order to formulate an inductive argument, finite order versions
of $(BA)$,$(FA)$ and $(EG)$ are introduced.
So, $\Omega_n$ stands for a set of parameters whose associated maps $f_a$
are supposed to inductively
satisfy $(BA)_n$, $(FA)_n$ and $(EG)_n$ where these assumptions are
formulated for each one of finitely many {\it critical approximations}
of order $n$ in  a set $\CC_n$.
By excluding some subintervals of $\Omega_n$ if necessary
we achieve a subset $\Omega_{n+1}$ where $(BA)_{n+1}$, $(FA)_{n+1}$ hold for a set $\CC_{n+1}$
of critical approximations of order $n+1$.
Further it is shown that these two last assumptions imply that $(EG)_{n+1}$ also holds,
recovering the induction hypothesis.
To recall a bit more precisely the structure of the arguments,
the sets $\CC_n$ and $\CC_{n+1}$ can be chosen exponentially close in $n$
and so each map $f_a$ with $a \in \Omega_\infty$ is associated to an infinite set
of ``true'' critical points satisfying (EG). See Section~\ref{ss:ind} for more details.

\paragraph{Probability of exclusions.}
As part of the induction argument
it is shown that parameters in the same connected component
$\omega$ of $\Omega_{n-1}$ have
critical orbits indistinguishable up to time $n-1$
implying distortion bounds of derivatives with respect to
phase and parameter spaces which allows
to  estimate
\[
\Leb(\Omega_{n}) \geq (1-\epsilon^n)\Leb(\Omega_{n-1})
\]
where $0< \epsilon < 1$ does not depend on $n$.

\section{Arguments in the cubic setting}

In the following sections we use a series of small constants
which are consistently much larger than $b$.
We use freely the convention of using $C>1$ as a generic large constant
not depending on $b$.
In the same spirit $0 < c < 1$ represents a small constant not depending
on $b$.

\subsection{Induction ingredients}\label{ss:hyp}

\subsubsection{The critical region}\label{ss:mane}

In this section we establish the constant $\delta$ in hypothesis of our
families in Section~\ref{s:fam}.

Recall we define linearizing neighborhoods in Section~\ref{s:fam}.
Let us suppose that $x_2$ is defined at least in the interval $(-1,1)$,
just in order to simplify notation.
Fixed a large integer $N>0$, there exists $\tilde{\delta} > 0$ such that
$\sigma_2^N \tilde{\delta} = 1$.
Note that any rectangle $\tilde{S} = [-\tilde{\delta}, \tilde{\delta}] \times
[-T,T] \subset \mathcal{U}_2$ remains in $\mathcal{U}_2$ for $N$ iterates.

Once and for all we fix $\delta = \tilde{\delta}/(10\sigma_1)$ and recall that $R_\delta =
[-\delta, \delta]^2$.
For $N$ large, $R_\delta \subset \mathrm{int} R_0$ (see $H_2$ for definition of $R_0$).

Now, fix $\eta = \delta^2$ in Section~\ref{s:fam}.
It is easy to see from the local form that
$\tilde{f}_{a,\delta_0} (R_\delta) \subset \tilde{S}$.
So, the assumptions above are compatible with $(H_3)$.

Hence, the appropriated choice of the constants $N, \delta, \eta, A_1, B_1$,
and $C_1$, permit us to construct families of diffeomorphisms satisfying
hypothesis $(H_1)$-$(H_4)$, where $f_{a,\theta}$ preserves the cone of width
$1/10$ and expands their vectors at points not in $R_\delta$.

We call $R_\delta$
the {\em critical region}.
This is a natural choice since the most dramatic effect of bending on
horizontal arcs inside $R_0$ occurs in a roughly vertical curve passing near
$q$.

Thus, the hypothesis stated on the families presented here yield hyperbolic behavior
outside $R_\delta$ as claimed in next lemma.

\begin{lemma}
There is $\sigma_0 > 1$ such that for any family $(f_{a,\theta})$ as above,
for any $z \notin R_\delta$ and every norm-1 vector $v = (v_1,v_2)$ with
$\slope(v) \leq 1/10$, we have
$$
\abs{ \slope Df_{a,\theta}(z)\cdot v } < 1/10
$$
and
$$
\| Df_{a,\theta}(z)\cdot v \| > \sigma_0.
$$
\end{lemma}

\subsubsection{Fields of contractive directions}\label{ss:stabf}

Again we write $f_{a,\theta}$ as $f$.
The derivative map $Df(z)$ define two orthogonal subspaces $E(z)$
and $F(z)$ on the tangent space $T_zM$ corresponding to the most
contracted and the most expanded ones. This is true under mild
hypothesis on the non-conformality of $Df(z)$ and is particularly
true in our setting.

These direction fields depend smoothly on the point $z$ and are
defined on some neighborhood of $z$ as long as the non-conformality
of the derivative holds. Integrating these fields we get two
orthogonal foliations $\mathcal{E}$ and $\mathcal{F}$. All this
reasoning can be reproduced for $Df^{k}(z)$, $k \ge 1$, if
non-conformality holds. The corresponding sequence of
finite order vector fields $E^{(k)}$ and $F^{(k)}$ as well as the
finite-order foliations
$\mathcal{E}^{(k)}$ and $\mathcal{F}^{(k)}$ can be thought of as
finite-versions of classical stable and unstable bundles and manifolds.
See \cite[Section 2.3]{LV03} for useful comments on this subject.

Let $1 < n < \nu$ and suppose the contractive fields
$E^{(n)}$ and $E^{(\nu)}$
are defined in an open set $U \subset M$.
These fields are almost constant and are exponentially  close
in $n$. These and other important facts are collected
in the next lemma. We write $\wzero = (1,0)$ and
given $\lambda > 0$ we say that a point $z=(x,y)$ is $\lambda$-expanding
up to time $n \geq 1$ if
\[
  \norm{Df^j(z)\wzero} \geq \lambda^j, \qquad \text{for all } 1 \leq  j \leq n.
\]

\begin{lemma}[Contractive fields] \label{lem:cd}
There exists $\tau > 0$ sufficiently small such that if
 $\hat{z}$ is $\lambda$-expanding up to time $n \ge 1$, for some $\lambda \gg b$ and $\xi$
satisfying $\dist(f^j(\xi),f^j(\hat{z})) < \tau^j$ for every $0\le j \le n-1$
then, for any point $z$ in the $\tau^n$-neighborhood of $\xi$ and
for every $1\le \ell \le k \le n$,
\begin{enumerate}
  \item $E^{(k)}(z)$ is uniquely defined and nearly
  vertical: $\abs{\slope (E^{(k)}(z))} \ge c/\sqrt{b}$;
  \item $\ang (E^{(\ell)}(z),E^{(k)}(z)) \le
  (Cb)^\ell$ and $\|Df^\ell(z)E^{(k)}(z)\| \le
  (C\sqrt{b})^\ell$;
  \item $\| D_* E^{(k)}(z)\| \le C \sqrt{b}$ and
  $\|D_*^2E^{(k)}(z)\| \le C\sqrt{b}$;
  \item $\| D_*(Df^\ell E^{(k)}(z))\| \le (Cb)^\ell$;
  \item $1/10 \le \|Df^n(\xi)\wzero\|\, / \, \|Df^n(z)\wzero \|
  \le 10$;
  \item $\ang(Df^n(\xi)\wzero,Df^n(x)\wzero) \le
  (\sqrt{C\tau})^n$.
\end{enumerate}
where $D_*$ stands indistinctly for derivatives with respect to $z,a$ or $\theta$.
\end{lemma}
\begin{proof}
Analogous to \cite[Section 5.3]{BC91},
\cite[Section 7C]{MV93}. See also
\cite[Section 2.1]{BeV01}, \cite[Section 2.3]{LV03} .
\end{proof}

\begin{rem}\label{rem:N}
For future reference, let us point out that
we can assume trivially constructed
contractive directions in the whole of $f(R)$
of all orders up to some large positive integer $N$:
every $z \in f(R)$
satisfies $\norm{Df^j(z)\cdot w_0} \ge \sigma_2^j$ for all $j \le N$.
Indeed this is related to the interval of time while the orbits of points in $R$
remains near the fixed point $Q$ and we can make $N$ as large as we want
taking the interval of $a$-parameters $[-\eta,\eta]$ sufficiently small.
Furthermore in the coordinate system introduced in Section \ref{s:fam}
these fields coincide with the vertical Euclidean foliation.
\end{rem}

\subsubsection{Critical points}\label{ss:cp}

We use the expression {\bf almost flat curve} to refer to the image of a
para\-me\-tri\-za\-tion $x \mapsto (x_0+x,y_0+y(x))$ with
$y,y',y''$ of $\mathcal{O}(\sqrt(b))$.


Let $\tilde{z}$ be the point in $W^u(P) \cap \partial R$ closest to $q$
in $W^u(P)$. We define
$G_0 = $  arc  $[P,\tilde{z}]$ in $W^u(P)$ and proceed by induction:
once defined $G_{n-1}$ we put $G_n = f(G_{n-1}) \setminus G_{n-1}$.

The set $G_n$ is called the {\bf arc of generation} $n$.

Given an almost flat curve $\gamma$ in $R$ we write $t(z)$ for the unit tangent
vector to the curve $\gamma$ in the point $z$.
Naturally there exists some $\theta'$ such that if $\theta < \theta'$
then there exists an unique point $z$ which minimizes
\begin{equation}\label{eq:ang}
    \measuredangle\big(t(f(z)), E^{(1)}(f(z))\big).
\end{equation}

This is an easy consequence of the cubic nature of
the definition of our families (see Section~\ref{s:fam}).

The same reasoning shows that under similar conditions there exists in $G_0$
an unique critical approximation $z_0^{(j)}$
of order $j$ for each $j$ from 1 to $N$.
Furthermore in view of Remark \ref{rem:N} we know that $z_0^{(j)}=\hat{q}$
for $j \leq N$.

Let us fix $\rho \gg b$ a small positive number and denote by $\gamma(z,\rho)$ the
$\rho$-neighborhood of $z$ in $W^u(P)$.
Suppose we have already defined critical approximations $z^{(1)},\ldots,z^{(n-1)}$
in an arc $\gamma$ of $W^u(P)$ of length at least $\rho^n$.
If in a neighborhood of $f(\gamma)$ we can define the contractive field $E^{(n)}$
then we can formulate the problem of minimizing an expression similar to (\ref{eq:ang}).
If this problem have an unique solution, we can define $z^{(n)}$.
This effectively works since from Lemma \ref{lem:cd} we know that
the angle between $E^{(n)}$ and $E^{(n-1)}$ is at most $(Cb)^n$
(and $b^n \ll \rho^n$).

When this process can be repeated for all $n \geq 1$ we will eventually define a
{\bf limit critical point} $z_\infty$.

It can be deduced from Lemma \ref{lem:cd} a natural algorithm to induce
critical approximations of arbitrary generation from
lower generations ones.
Let us consider  $\gamma_1=\gamma(z_1,\ell)$ and $\gamma_2=\gamma(z_2,\ell)$ arcs
of $W^u(P)$ with $\ell \geq \rho^n$.
We assume that $z_1$  is a  critical approximation of order $n$. Hence if
$\dist(z_1,z_2) \ll \rho^n$ then it is easy to see that $\gamma_2$ also contains a
critical approximation of order $n$.


Notwithstanding the choice of minimizing angles between contracting foliations
and unstable manifolds is nothing but a natural one, we have to deal with the
lackness of meaning of this notion when $\theta$ is very small.
Recall that Lemma~\ref{lem:cd} implies that whenever the contractive directions
do exist they converge exponentially fast, even with respect to parameters,
whereas our local form promotes a fast bending of the unstable manifold while
varying $\theta$. So we have to design careful ways of detecting, or preventing,
configurations like that shown in Figure~\ref{f.bad}.
This issues will be addressed in Subsection~\ref{ss:tang}.

\begin{figure}[phtb]
\centering
\psfrag{R}{$R$}
\psfrag{r}{$\hat{q}$}
\psfrag{l}{$\Phi_{a,\theta}$}
\psfrag{q}{$q+(a,0)$}
\includegraphics[height=4cm]{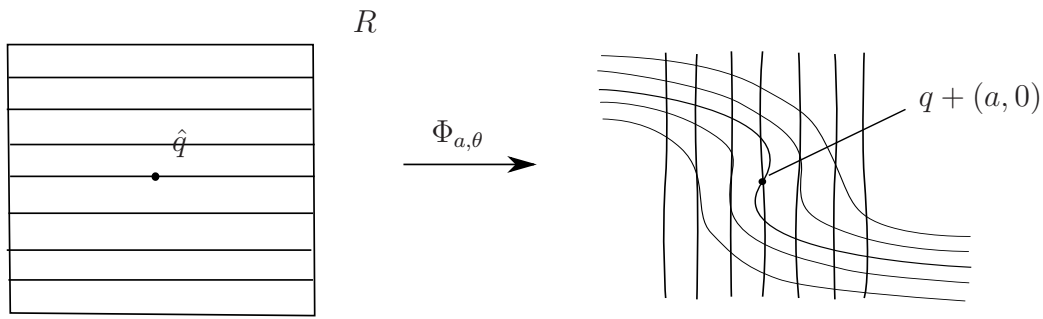}
\caption{Bad situation} \label{f.bad}
\end{figure}


\subsubsection{Induction procedure in the cubic setting}\label{ss:ind}

We are going to outline some aspects of the induction procedure which are
relevant in our context.
A substantial part of the arguments in the original works are based solely
in the strong dissipativeness of the system and are straightforward
applicable to this setting.

As a matter of notation we write $\wzero(z_0) = \wzero$ and

\[
\wn(z_0) = Df^n(f(z_0))\cdot\wzero(z_0).
\]

Let us suppose we have defined for $1 \leq k \leq n$ finite sets $\CC_k$
of critical approximations of order $k$.
We collect here some facts that are assumed to be true
for each
$z_0 \in \CC_k$ for a given $k$.

The generation $g$ of $z_0$ is much smaller than $k$  and
$z_0$ is the center of an almost flat piece of $W^u(P) \cap R_\delta$
of length at least $2\rho^g$.

We have exponential growth of derivatives:

\bigskip
\noindent $\textrm{(EG)}_k \hspace{4cm} \norm{\wj(z_0)} \geq e^{c j}, $ \;\;
$0 \leq j < k$.
\bigskip

The orbit of $z_0$ is divided into pieces as follows.
Given $0 \leq \nu \leq k$ if $z_\nu=f^\nu(z_0)$ is inside $R_\delta$
we say that $\nu$ is a {\bf return } time for $z_0$.
For every such return occurring up to time $k$ we have an associated
element $\xi=\xi(z_\nu)$ of $\CC_k$ which is called the {\bf bind critical
point} of $z_\nu$. We define
\[
 d_\nu(z_0) = \norm{z_\nu - \xi}
\]
and assume that

\bigskip
\noindent $\textrm{(BA)}_\nu \hspace{4cm} d_\nu(z_0) > e^{-\alpha \nu}$
\bigskip

\noindent holds and also
\begin{equation}
\label{eq:tp}
\abs{ \slope(\bfw_{\nu-1}) - \slope(t(\xi)) } \ll d_\nu(z_0).
\end{equation}

Following a return $\nu$ there are $p$ iterates called the {\bf bound period}
of $z_\nu$ where $p$ is defined in order to maximally satisfy (here $h$ is a
constant fixed a priori):
\[
\dist(z_{\nu+j},\xi_j) \le h e^{-\beta j},\qquad 0 \leq j \leq p.
\]

Each bound period starts with a segment called {\bf folding period} with size
$\ell$ satisfying
\begin{equation}\label{eq:fold}
 \ell \approx \frac{C}{\log(1/b)}\log(1/d_\nu(z_0)^2) \ll p .
\end{equation}

Each
iterate $z_j$ which is not part of a folding period is said to be a
{\bf fold free} iterate. If $z_j$ is fold free then
\[
  \abs{\slope(\wjp)} \leq C\sqrt{b}.
\]

Returns can occur before the end of the bound period of previous returns but
bound periods are nested: if $\nu_1$ and $\nu_2$ are successive returns whose
bound periods have lengths of $p_1$ and $p_2$ iterates respectively and if
$\nu_1 + p_1 \geq \nu_2$ then $\nu_2 + p_2 < \nu_1+p_1$.
Returns occurring outside any bound periods are referred to as {\bf free
returns}.

If $0 < \nu_1 < \nu_2 < \cdots < \nu_m \leq k$ are the free returns of
$z_0$ up to time $k$ and $p_j$ is the length of the bound period associated
to the return $\nu_j$ then

\bigskip
\noindent $\textrm{(FA)}_k$ \hspace{4cm}
 $FA(z_0,k) = \displaystyle{\sum_1^m p_j \leq \alpha k}$.
\bigskip

Let $\gamma: s \mapsto (x_0 + s,y_0 + y(s))$ be an almost flat curve in
the critical region whose image under $f$ is contained in an open set where
the contractive field  of order $k$ is defined.
A tool which is an important part
of the arguments is the splitting algorithm
which we describe now in an informal way.

While defining the notion of critical approximation of some finite order
we considered the action of our dynamical system over
almost horizontal arcs in the critical region.
Strong dissipativiness of the system and hyperbolicity outside the critical region
yields the important geometric fact that the iterates of $\wzero$ return almost
horizontal.
But if the parameter $\theta$ is very small and $n$ is a return time for $z_0$
then $\wn(z_0)$ can be too close to the contractive direction.
To estimate the loss of growth we split this vector in the horizontal direction
and the direction of the contractive foliation.
The magnitude of the horizontal component is related to the distance $d=d_n(z_0)$
of the return with respect to the binding critical point.
In the Hénon-like case this component is of magnitude $d^2$.
%
%
%
We now proceed to investigate these magnitudes in our context. Recall that we have
\begin{align}
\Phi_{a,\theta}(x,y) =
\left[
\begin{array}{c}
a +  \sigma_1 y + x(- A_1 \theta + B_1 x^2 + C_1 y^2)\\
- b x + y(- A_2 \theta + B_2 x^2 + C_2 y^2) \\
\end{array}
\right]
\end{align}
and
\begin{align}
D\Phi_{a,\theta}(x,y) =
\left[
\begin{array}{cc}
A(x,y) & B(x,y)\\
C(x,y) & D(x,y)
\end{array}
\right],
\end{align}
with
\begin{align*}
A(x,y) & = -A_1\theta + 3 B_1x^2 + C_1 y^2 , \\
B(x,y) & = \sigma_1 + 2C_1xy ,\\
C(x,y) & = b + 2B_2xy , \text{ and} \\
D(x,y) & = -A_2\theta + B_2x^2 + C_2 y^2.
\end{align*}

Writing $A(s) = A(\gamma(s))$ and similar expressions for $B,C$,
and $D$, we have
\begin{align*}
A'(s) & = 6 B_1 s + 2 C_1 yy', \\
B'(s) & = 2C_1y + 2C_1 sy' = 2C_1(y+sy'), \\
C'(s) & = 2B_2(y + sy'), \\
D'(s) & = 2B_2s + 6C_2 yy',\\
A''(s) & = 6B_1 + 2C_1((y')^2+yy'), \\
B''(s) & = 2C_1y' + 2C_1(y'+sy'') = 2C_1(2y'+sy''), \\
C''(s) & = 2B_2(2y'+sy''), \text{ and} \\
D''(s) & = 2B_2 + 6C_2((y')^2+yy').
\end{align*}

Suppose in a neighborhood of $f(\gamma)$ the existence of the contractive
direction of order $n$ and consider the field $\ve(s) = \ve^{(n)}(s) =
(q(s),1)$, collinear with that direction.
We can suppose $\abs{q},\abs{q'},\abs{q''}$ bounded by $C\sqrt{b}$.
Writing $\vt(s) = Df(\gamma(s))\cdot\gamma'(s)$ it follows that
\begin{equation} \label{e.ts2}
\vt(s) = \alpha(s) \ve(s) + \beta(s) \wzero.
\end{equation}
So, the previous estimates imply (the prime means the derivative
with  respect to $s$)
\begin{align*}
\alpha & = C + Dy, \\
\alpha' & = C'+D'y+D''y, \text{ and}\\
\alpha'' & = C'' + D''y +D'y'+D'y' + Dy''.
\end{align*}
It is immediate that $|\alpha|,
|\alpha'|$, and $|\alpha''|$ are all bounded above by $C\sqrt{b}$.
We also have
\[
\beta  = A + By - \alpha q \\
\]
and so
$$
\abs { \beta'(s)  - 6 C_1 s } \leq  C\sqrt{b} \qquad \text{ and }  \qquad
\abs { \beta''(s) - 6 C_1 } \leq  C\sqrt{b}.
$$

Now suppose that  $\gamma(0)$ is a critical point of order $m \geq k$.
Let $\beta$ be  obtained as before from the splitting with
respect to $k$-contractive directions and let us consider also $\tilde\beta$
as the corresponding function while splitting with respect to $m$-contractive
directions.
It is an easy consequence of Lemma~\ref{lem:cd} that
\begin{equation}\label{eq:beta0}
 \abs{ \beta'(0) } \approx (Cb)^m \quad \text{and} \quad \tilde\beta'(0) = 0.
\end{equation}
Also, in view of (\ref{eq:fold}) we get
\[
 (Cb)^m \leq d_n(z)^2
\]
and hence
\begin{equation}\label{eq:mybeta}
 \abs{\beta(s) - \beta(0)} \approx K_1 d_n(z)^2,
\end{equation}
for some $K_1 > 1$ large.

\begin{lemma}\label{lem:boundp}
Let $n$ be a return for $z=z_0 \in \CC_k$, with $n \leq k$, and let $p$
be the length of the corresponding bound period.
Let $\xi$ be the associated binding critical point and $d = \dist(z_n,\xi)$.
Then
\begin{itemize}
 \item[(a)] $p \approx \log(1/d) $
 \item[(b)] $\norm{\text{\bf w}_{n+p}(z_0)} \geq e^{\tilde{c}(p+1)} \norm{\text{\bf w}_{n-1}(z_0)}$
\quad for some fixed $\; \tilde{c} > 1$.
\end{itemize}
\end{lemma}

\begin{proof}
These facts are derived much in the same way as in
\cite[Section~7.4]{BC91} and \cite[Section~10]{MV93}.
We outline some steps where estimates are affected by the cubic setting.

We suppose that there is an almost flat curve $\gamma$ as above with $\gamma(0)=\xi$.
For some $\hat{s} \approx d$ we have $\gamma(\hat{s})=z_n$ and $\gamma'(\hat{s})$
is collinear with $\wnp(z_0)$ (and almost unitary).
In the sequel we write $\wj(s) = \wj(\gamma(s))$, for $j \leq n$.

In order to show that in the bound period we have a kind of distortion
bound which permits deriving growth of derivatives $\wj(s)$ from the
induction hypothesis of growth of $\wj(0)$, first one shows that it is
possible to write
\begin{equation}\label{eq:mv}
 \wj(s) = \lambda(s)(\wj(0) + \epsilon(s)),\quad c \leq \lambda(s) \leq C
\quad\text{and}\quad
\norm{\epsilon(s)} \ll \norm{\wj(s)}
\end{equation}
from which we get
\[
 \norm{ \wj(s) } \approx \norm{ \wj(0) } \geq e^{cj} \quad(\text{by induction}).
\]


As in (\ref{e.ts2}) we use the contractive field $\text{\bf e} = \text{\bf e}^{(p)}$
of order $p$, the binding period associated to this return, and write

\begin{equation} \label{e.ts3}
\vt(s) = Df(\gamma(s))\cdot\gamma'(s) = \alpha(s) \ve(s) + \beta(s) \wzero.
\end{equation}

Hence we can estimate how much a point $\gamma(s)$ gets far from its binding
point $\xi=\gamma(0)$ in the next iterates by writing
\[
 f^{j+1}(\gamma(s)) - f^{j+1}(\gamma(0)) = \int_0^s Df^j(f(\gamma(s)))\cdot\vt(s)\; ds
\]
and using (\ref{e.ts3}) we get
\[
 \alpha(s)\ve_j(s) + \beta(s)\wj(s) =
 \alpha(s)\ve_j(s) + (\beta(s)-\beta(0))\wj(s) + \beta(0)\wj(s),
\]
with subscripts $j$ meaning the obvious iterate under action of $Df$.

In particular, for $j=p$, since $\ve(s)$ is exponentially contracted for $p$
iterates we know that the integrand in the first term on
the right hand side is of magnitude less than $(Cb)^p$.
%
From (\ref{eq:mybeta}) and writing $\widehat\Theta = \beta(0)$ we get the
estimate
\[
e^{-\beta p} \approx  \dist( z_{n+p} , \xi_p) \approx \norm{ \wpp(0) }
(\widehat\Theta d + K_2d^3),
\]
with $K_2 = (1/3)K_1$.
Taking into account the definition of $p$ and that
$\norm{ \wpp(0) } \geq e^{cp}$ we get
\begin{equation}\label{eq:ecp1}
 e^{-\beta (p+1)} e^{-c} \leq e^{cp} (\widehat\Theta d + K_2d^3) \leq e^{-\beta p}.
\end{equation}

\begin{rem}
Note that more rigorously we must write $\widehat\Theta = \beta(\tilde{s})$ where
$\beta'(\tilde{s}) = 0$, but since $\abs{\tilde{s}} \leq (Cb)^p$ we have
$\abs{\beta(0)} \approx \abs{\beta(\tilde{s})}$ and $d \approx s \approx \tilde{s}$.
\end{rem}

We now use the fact that $\widehat\Theta > 0$ along all our construction, as will
be explained in Section~\ref{ss:tang}. The last result gives easily item (a).
Furthermore  the second inequality in (\ref{eq:ecp1}) gives
\[
 e^{cp} K_2 d^3 \leq e^{-\beta p}
\]
which implies
\[
 \frac{1}{d} \geq e^{\frac{1}{3}(c + \beta) p} K_2 ^{\frac{1}{3}}.
\]
From here and the first inequality in (\ref{eq:ecp1}) we get
\begin{equation}\label{eqp0}
e^{cp} (\widehat\Theta  + K_1 d^2) \geq
 e^{-\beta (p+1)} e^{-c} e^{\frac{1}{3}(c + \beta) p} K_2 ^{\frac{1}{3}}.
\end{equation}
On the other hand note that
\begin{equation}\label{eqp1}
{ \norm{ \bfw_{n+p}(z_0) } \over  \norm{ \bfw_{n-1}(z_0) } }
\geq (1+C\sqrt{b})\norm{ Df^{p+1}(\gamma(\hat s) )\cdot \gamma'(\hat s) }
\end{equation}
and exploring again (\ref{e.ts3}) we get
\begin{equation}\label{eqp2}
\norm{ Df^{p+1}(\gamma(\hat s) )\cdot \gamma'(\hat s) }
\geq
\abs{\beta(0)} \norm{\bfw_p(0)} - C\sqrt{b}(Cb)^p
\end{equation}
and we also have
\[
 \abs{ \beta(0) } \approx \widehat\Theta  + K_1d^2
\qquad \text{and} \qquad
 \norm{ \bfw_p(0) } \geq e^{cp}.
\]
Combining (\ref{eqp0}),(\ref{eqp1}) and (\ref{eqp2}) we get
immediately item (b).
\end{proof}

\begin{rem}
We can derive exponential growth of $\bfw_k(z_0)$ by observing that
\[
 \norm{\bfw_k(z_0)} =
\prod_{i=1}^{k} { \norm{\bfw_i(z_0)} \over \norm{\bfw_{i-1}(z_0) } }.
\]
For each $n \leq k$ a (free) return time with corresponding bound period $p$ we have
\[
\prod_{i=n}^{n+p} { \norm{\bfw_i(z_0)} \over \norm{\bfw_{i-1}(z_0) } }
=
{ \norm{\bfw_{n+p}(z_0)} \over \norm{\bfw_{n-1}(z_0) } } \geq 1
\]
according the previous lemma and
\[
\prod_{i=n+p+1}^{\nu} { \norm{\bfw_i(z_0)} \over \norm{\bfw_{i-1}(z_0) } }
=
{ \norm{\bfw_{\nu}(z_0)} \over \norm{\bfw_{n+p}(z_0) } } \geq e^{c\mu}
\]
where $\nu$ is the next free return after $n$ and $\mu = \nu - (n + p + 1)$
and this last estimate follows from the results in Section~\ref{ss:mane}.
Those parameters satisfying the induction hypothesis in particular
obey $(FA)_k$ and this gives the expected growth.
\end{rem}

\subsection{Creation of tangencies in a controlled  way}\label{ss:tang}

We write $\mathcal{R} = [-\eta,\eta] \times [-1,\epsilon]$
for the $(a,\theta)$-parameter space.
Recall that in the one dimensional case (see \cite{HMS07})
we deal with curves at the parameter space which can be
described by maps $a \mapsto (a,\tilde{\theta}) \in \mathcal{R}$, for
$\tilde{\theta}$ fixed, and apply the {\em exclusion parameter
arguments} (which we will denote as \textsf{EPA} from now on)
to them.
Here we generalize this notion.

\begin{definition}\label{def:tfcurve}
A {\em $\theta-$flat curve} is the graph in $\mathcal{R}$
of a smooth map $[-\eta,\eta] \to [-1,\epsilon]$,
which
has all derivatives up to order 3 bounded by $C\sqrt{b}$.
\end{definition}

Note that it is possible to extend easily the arguments
and apply \textsf{EPA} to \tf curves.

We also want to introduce a notion which will indicates how far we are from
``forming tangencies''. Let $\Upsilon$ be a \tf curve as above.
While applying \textsf{EPA} to this curve, let
$(a,\theta) \in \Upsilon$ be a parameter that has not been excluded
up to time $j$.
In the context of the arguments we are dealing with among several facts
this means that there exists a finite
set $\mathcal{C}_j$ of critical approximations of order $j$ associated to the map $f_{a,\theta}$
and each one of these critical approximations lies on a sufficiently large
and flat arc of $W^u(P_{a,\theta})$.
Moreover around the images of these arcs
there are well defined maximal contractive directions of order $j$.
Since the critical approximations are defined intrinsically
as minimizing the angles between $W^u$ and the contractive
foliations (see (\ref{eq:ang})),
we can define
\[
\measuredangle_j(a,\theta) = \min \{ \measuredangle(t(\xi), \Gamma^j(f(\xi))\;; \xi \in \mathcal{C}_j \}.
\]
where $\Gamma^j(\cdot)$ is the leave
of the contractive foliation passing through the specified point.


We fix a small number $\varrho$, but satisfying $\varrho \gg b$.
Given $m > 1$ it follows from Remark~\ref{rem:N} and Section~\ref{ss:cp}
that for each $1 \leq j \leq N$ there exists a \tf curve
$\Upsilon_j^m$ in $\mathcal{R}$ satisfying,
\[
\measuredangle_j(a,\theta) = \varrho^m \qquad \text{ for } (a,\theta) \in \Upsilon_j^m.
\]

\begin{rem} In fact $\theta$ is constant over $\Upsilon_j^m$, for $j \leq N$.
\end{rem}

\begin{rem}
While discussing the splitting algorithm in Section~\ref{ss:ind}
in a number of places we have assumed that the minimum value of $\beta(\cdot)$
along an almost flat curve passing through a critical point (of finite order)
was attained at that point with an associated value denoted by $\widehat\Theta$.
The context presented now justify why we could assume $\widehat\Theta > 0$.
\end{rem}

Note that if $(a,\theta)$ was not excluded up to time $j$ then for each $(\tilde{a},\tilde{\theta})$
sufficiently close we also have
the $j$-contractive field defined and the
local form (see Section~\ref{s:fam}) gives
\begin{equation}\label{eq:varang}
\abs{ \measuredangle_j(a,\theta) - \measuredangle_j(\tilde{a},\tilde{\theta})}
\leq c_1\sqrt{b}\abs{ a - \tilde{a}} + C_2\abs{ \theta - \tilde{\theta} }.
\end{equation}
In particular we get a family of \tf curves
\[
\{ \Upsilon_N^m \}_{m\geq 1}
\]
in $\mathcal{R}$ satisfying, $\measuredangle_j = \varrho^m$.

Let us introduce some notations. While applying \textsf{EPA} to
$\Upsilon_j^m$ we get a collection
$\Omega_k = \Omega_k(j,m)$ of subsets of $\Upsilon_j^m$ satisfying
\[
\Upsilon_j^m = \Omega_0 \supset \Omega_1 \supset \Omega_2 \cdots \supset \Omega_k \supset \cdots
\]
Each $\Omega_k$ is an union of subarcs $\omega$ of $\Upsilon_j^m$ collected in a partition
$\PP_k = \PP_k(j,m)$ of this curve.

For each $m > N$ we want to exhibit a \tf curve $\Upsilon^m =\Upsilon_m^m$
such that $\measuredangle_j(a,\theta) \approx \varrho^m$ for all
($a,\theta) \in \Omega_j(m) := \Omega_j(m,m)$, that is to say, for all
$(a,\theta) \in \Upsilon^m$ not excluded by \textsf{EPA} up to time $j$.
Here the scale $\approx$ is chosen to precludes presence of tangencies
and will be made precise in a little while.

The curve $\Upsilon^m$ will be constructed from $\Upsilon_N^m$
by applying \textsf{EPA} a finite number of times and a {\bf correction
procedure} (to be detailed) in order to get
\[
 \Upsilon_N^m \rightsquigarrow
 \Upsilon_{N+1}^m \rightsquigarrow
 \Upsilon_{N+2}^m \rightsquigarrow
\cdots
 \Upsilon_{m-1}^m \rightsquigarrow
 \Upsilon_{m}^m  = \Upsilon^m
\]

We describe this construction and formulate precisely
the correction procedure inductively. Suppose we have
just constructed $\Upsilon_j^m$, with $N \leq j \leq m-1$
satisfying:
\begin{equation}
 \label{eq:upsjm}
\abs{\measuredangle_j(a,\theta) - \varrho^m } \leq (Cb)^j
\qquad \text{ for all } (a,\theta) \in \Omega_j(j,m)
\end{equation}

By induction we are able to apply the induction of \textsf{EPA}
to $\Upsilon_j^m$ up to time $j$. For each $1 \leq k \leq j$
the  partition $\PP_k = \PP_k(j,m)$ of subarcs $\omega$ of $\Upsilon_j^m$
is obtained first refining $\PP_{k-1}$ and then throwing away
some of its arcs by the exclusion rules of \textsf{EPA}.
Let $\omega$ be one of the arcs of $\PP_{j}$.
In the most general case there is a sub-arc $\omega_{exc} \subset \omega$
which must be excluded while passing from time $j$ to $j+1$. Let
$\omega'$ be a connected component of $\omega \setminus \omega_{exc}$.
Note that on $\omega'$  by definition we have
\begin{equation}\label{eq1}
 \abs{ \measuredangle_j - \varrho^m} < (Cb)^j
\end{equation}
and by Lemma~\ref{lem:cd}
\begin{equation}\label{eq:antespert}
\abs{\measuredangle_{j+1} - \measuredangle_j } < (Cb)^j
\end{equation}

We claim that near $\omega'$ we can find another arc of \tf curve $\omega''$
on which we have
\begin{equation}\label{eq2}
 \measuredangle_{j+1} = \varrho^m
\end{equation}

This can be deduced from the following facts.
The contractive directions are ($C\sqrt{b}$) almost constant
with respect to the phase and parameter space (see Lemma~\ref{lem:cd})
where they are defined.
Also the dynamics of $f(a,\theta,z)$ and $f(\tilde{a},\tilde{\theta},\tilde{z})$
are
indistinguishable up to time $j+1$ if the distance between
$(a,\theta,z)$
and
$(\tilde{a},\tilde{\theta},\tilde{z})$ is at most of order $C(\sqrt{b})^{j+1}$
and, in view of (\ref{eq:varang}), for compensating the terms of $\mathcal{O}(b^j)$
in (\ref{eq:antespert}) we only need to perturb $\theta$ in a order of magnitude
less than $\mathcal{O}(b^j)$.

Finally, we define $\Upsilon_{j+1}^m$ as a \tf curve containing each one
of those corrected arcs.
In order to be more precise about the existence of such a curve we
recall that each $\omega_{exc}$ has length of order of $Ce^{-\alpha j}$ which
is $\gg$ than the distance from $\omega'$ to its corrected version
$\omega''$.
Note also that for each arc of $\PP_{j+1}(j+1,m)$ we can not be sure that
(\ref{eq2}) still holds but the same arguments above
give surely that
\[
 \abs{ \measuredangle_{j+1} - \varrho^m} \leq (Cb)^{j+1}
\]
which recovers exactly the induction hypothesis of (\ref{eq:upsjm}).

Note that this strategy yields a \tf curve $\Upsilon^m = \Upsilon_m^m$
satisfying
\begin{itemize}
 \item[] If $(a,\theta) \in \Omega_k(m)$, $k \geq m$ then
 $\abs{ \measuredangle_k(a,\theta) - \varrho^m} < (Cb)^m \ll \varrho^m$.
\end{itemize}
%
Hence for each $(a,\theta)$ in the positive Lebesgue measure set $\Omega_\infty(m)$
we have

\[
 0 < \abs{ \measuredangle_\infty(a,\theta)} \approx \varrho^m.
\]


\section{Proof of Theorem~\ref{teo.main}}

In the previous section we construct subsets $\Omega_\infty(m)$ of the
parameter space.
According to Definition \ref{def:tfcurve} and since
the curves $\Upsilon^m$ and $\Upsilon^{m+k}$ are exponentially close on
$m$ for all $k$, we conclude that the family of smooth functions on
$[-\eta,\eta]$ whose graphs correspond to the family of curves
$(\Upsilon^m)_m$ converge uniformly to a continuous function and so
it is well defined the limit
$$
\widehat{\Omega}_\infty = \lim_m \Omega_\infty(m).
$$
The construction of each set $\Omega_\infty(m)$
is based on the exclusion parameters argument
whose one of the most relevant feature is that each
such final set has
positive measure and, in fact, similar to \cite{HMS07}, the construction
is uniform on $\theta$-parameters and so the measures of
$\Omega_\infty(m)$ are uniformly bounded away from zero, for all $m$.
These considerations yields the conclusion that
 $\widehat{\Omega}_\infty$ has positive measure.

The main feature of a parameter $(a,\theta) \in \widehat\Omega_\infty$
is the non hyperbolicity of the correspondent map $f_{a,\theta}$\;:
there is a point $\hat{z}$ in $W^u(P)$ such that the tangent direction
of $W^u(P)$ at $\hat{z}$ is mapped by the derivative of $f_{a,\theta}$
on a contractive direction.
So, it can not be hyperbolic once a direction is both forward and
backward exponentially contracted.
Thus, each parameter $(a,\theta) \in \widehat\Omega_\infty$ corresponds
to a non hyperbolic map $f_{a,\theta}$. By construction,
$\widehat\Omega_\infty$ is accumulated by $\Omega_\infty(k)$.
The fact that $f_{a,\theta}$, $(a,\theta) \in \widehat\Omega_\infty$,
is accumulated by hyperbolic maps $f_{a_k,\theta(a_k)}$,
$(a_k,\theta(a_k)) \in \Omega_\infty (m)$ (see
Corollary~\ref{c.exp-gr}) implies that $f_{a,\theta}$ belongs to the
boundary of $\HS$.

In order to prove Theorem~\ref{teo.main} we show in
Corollary~\ref{c.lyap-exp} the hyperbolicity of periodic points of
$f_{a,\theta}$, $(a,\theta) \in \widehat\Omega_\infty$.

In Proposition~\ref{p.non-per}, we show the existence of a positive
measure subset of $\widehat\Omega_\infty$ such that the unstable
manifold of $P$ is tangent to a stable manifold not associated to
periodic points.

From here we get that every point in the neighborhood $\mathcal{U}$
of $\Lambda$ (see Section~\ref{s:fam}) expands the horizontal direction:

\begin{prop}
\label{p.exp-gr}
Let $(a,\theta)$ be in $\Omega_\infty(m)$, $m \ge N$. There exists
$\hat{c} > 0$ and $\sigma > 1$ such that for every
$z \in \mathcal{U}$, we have
$$
\abs{Df_{a,\theta}^k(z)\cdot (1,0)} \ge \hat{c} \sigma^k.
$$
\end{prop}

\begin{proof}
Analogous to \cite[Proposition 7.1]{HMS07}.
\end{proof}

\begin{cor}
\label{c.exp-gr}
For all $(a,\theta)$ in $\Omega_\infty(m)$, $m \ge N$, the map $f_{a,\theta}$
is (uniformly) hyperbolic.
\end{cor}

\begin{proof}
Recall that the expansion of some direction besides the global strong
dissipativeness of $f_{a,\theta}$ imply the existence of a $Df_{a,\theta}$
invariant splitting of the tangent space of the neighborhood of $\mathcal{U}$
in contractive and expanding subbundles.
\end{proof}

\begin{prop}
\label{p.lyap-exp}
Let $(a,\theta)$ be in $\widehat\Omega_\infty$. For (Lebesgue)
almost every $z\in \Lambda_{a,\theta}$, including all periodic
points,
$$
\lambda_{\inf}(z) = \liminf_{n\to \infty} \frac{1}{n+1}
\log \abs{Df_{a,\theta}^n(z)\cdot (1,0)} > 0.
$$
\end{prop}

\begin{proof}
Analogous to \cite[Proposition 7.2]{HMS07}.
\end{proof}

\begin{cor}
\label{c.lyap-exp}
Let $(a,\theta)$ be in $\widehat\Omega_\infty$. All periodic point of
$f_{a,\theta}$ are (uniformly) hyperbolic.
\end{cor}
\begin{proof}
The fact that the orbit of a periodic point is finite
associated to the claim of the previous proposition
gives the
expansion of the derivative along the horizontal direction. Again, by
the strong dissipativeness of the jacobian, it follows the existence of
a contractive direction.
\end{proof}

Finally, we prove the existence of a positive measure subset where
Theorem~\ref{teo.main} holds.
\begin{prop}
\label{p.non-per}
There is a positive measure subset $\widehat{\mathcal{A}}$ of
$\widehat\Omega_\infty$ such that for all $(a,\theta)$ in
$\widehat{\mathcal{A}}$ there exists a stable manifold $W^s$ of
$f_{a,\theta}$ not associated to a periodic point such that
$W^u(P)$ and $W^s$ are tangent.
\end{prop}

\begin{proof}
It follows from the control of recurrence of each critical orbit
that all critical points are non periodic. In particular, the point
of tangency $t=t_{a,\theta}$, $(a,\theta) \in \widehat\Omega_\infty$
is non periodic.

Moreover, if the tangency point $t$ belongs to a stable manifold of a
periodic point, then for each neighborhood $B(t)$ fixed,
there are a finite number of forward iterates inside $B(t)$.

Due the hyperbolicity outside a fixed neighborhood of $t$ and the
bounded distortion between phase space and parameter space, we have that
the set of parameters such that the forward orbit of the tangency point
goes into $B$ just a finite number of time has zero measure.
For $n$ large enough, let $\mathcal{B}_n$ be the set of parameters such
that the forward tangency orbit never goes to the ball of radius $1/n$
centered in $t$. Note that the set of parameters such that the tangency
orbit belongs to a stable manifold of a periodic orbit is contained in
$\cup \mathcal{B}_n$ which has zero measure. Hence, the subset of
$\widehat\Omega_\infty$ whose tangency belongs to a periodic point has
zero measure. This proves the proposition.
\end{proof}

To conclude the proof of Theorem~\ref{teo.main}, take
$$
\mathcal{A} = \left\{ a \in [\eta,-\eta] \colon (a,\theta) \in \widehat{\mathcal{A}}
\right\}.
$$
Since $\widehat{\mathcal{A}}$ is a limit set of positive measure subsets of
$\Omega_\infty(m)$ whose measures are uniformly bounded away from zero
and these sets are $\theta$-flat curves, we conclude
that $\mathcal{A}$ has positive measure, finishing the proof of Theorem~\ref{teo.main}.


\bigskip

\flushleft

{\bf Vanderlei Horita} (vhorita\@@ibilce.unesp.br)\\
Departamento de Matem\'{a}tica, IBILCE/UNESP \\
Rua Crist\'{o}v\~{a}o Colombo 2265\\
15054-000 S. J. Rio Preto, SP, Brazil

\bigskip

\flushleft

{\bf Nivaldo Muniz} (nivaldomuniz\@@gmail.com)\\
Departamento de Matem\'{a}tica, UFMA \\
Avenida dos Portugueses, S/N \\
65000-000 S\~{a}o Lu\'{\i}s, MA, Brazil

\bigskip

\flushleft

{\bf Paulo Rog\'{e}rio Sabini} \\
Instituto de Matem\'atica e Estat\'{\i}stica, UERJ \\
Rua S\~{a}o Francisco Xavier, 524 \\
20550-900 Rio de Janeiro, RJ, Brazil

\end{document}